\documentclass[10pt]{article}
\usepackage{array}
\usepackage{setspace}
\usepackage{color}
\usepackage[usenames,dvipsnames]{xcolor}

\usepackage[round]{natbib}
\usepackage{fullpage}

\usepackage{amsfonts}

\usepackage{graphicx}

\usepackage{mathrsfs}

\usepackage{amsmath}

\usepackage{amssymb}

\usepackage{amsthm}

\usepackage{float}

\usepackage{hyperref}

\usepackage{comment}

\usepackage[latin1]{inputenc}
\usepackage{amsmath}
\usepackage{amsfonts}
\usepackage{amssymb}
\usepackage{graphicx}
\usepackage{comment}

\newtheorem{theorem}{Theorem}[section]
\newtheorem{lemma}[theorem]{Lemma}

\newtheorem{definition}[theorem]{Definition}

\begin{document}
\renewcommand{\baselinestretch}{1.2}
\markright{
}
\markboth{\hfill{\footnotesize\rm Karl Rohe, Tai Qin, Haoyang Fan
}\hfill}
{\hfill {\footnotesize\rm RMLE on SBM} \hfill}
\renewcommand{\thefootnote}{}
\let\thefootnote\relax\footnote{\textbf{Acknowledgements:} Thank you to the anonymous referees and the associate editor for their thoughtful comments that have added to the quality of this research.  Research supported in part by NIH Grant EY09946, NSF Grant DMS-0906818, and grants from WARF.}
$\ $\par
\fontsize{10.95}{14pt plus.8pt minus .6pt}\selectfont
\vspace{0.8pc}
\centerline{\large\bf The Highest Dimensional Stochastic Blockmodel}
\vspace{2pt}
\centerline{\large\bf with a Regularized Estimator}
\vspace{.4cm}
\centerline{Karl Rohe, Tai Qin and Haoyang Fan}
\vspace{.4cm}
\centerline{\it Department of Statistics, University of Wisconsin Madison}
\vspace{.55cm}
\fontsize{9}{11.5pt plus.8pt minus .6pt}\selectfont

\begin{quotation}
\noindent {\it Abstract:}
In the high dimensional Stochastic Blockmodel for a random network, the number of clusters (or blocks) $K$ grows with the number of nodes $N$.  Two previous studies have examined the statistical estimation performance of spectral clustering and the maximum likelihood estimator under the high dimensional model; neither  of these results allow $K$ to grow faster than $N^{1/2}$.  We study a model where, ignoring $\log$ terms, $K$ can grow proportionally to $N$.  Since the number of clusters must be smaller than the number of nodes, no reasonable model allows $K$ to grow faster; thus, our asymptotic results are the ``highest" dimensional.     To push the asymptotic setting to this extreme, we make additional assumptions that are  motivated by empirical observations in physical anthropology \citep*{dunbar1992neocortex}, and an in depth study of massive empirical networks \citep*{leskovec2008statistical}.  Furthermore, we develop a regularized maximum likelihood estimator that leverages these insights and we prove that, under certain conditions, the proportion of nodes that the regularized estimator misclusters converges to zero.  This is the first paper to explicitly introduce and demonstrate the advantages of statistical regularization in a parametric form for network analysis. \par

\vspace{9pt}
\noindent {\it Key words and phrases:}
Stochastic Block Model, regularization, clustering, consistency, high dimensional
\par
\end{quotation}\par

\fontsize{10.95}{14pt plus.8pt minus .6pt}\selectfont
\section{Introduction}

Recent advances in information technology have produced a deluge of data on complex systems with myriad interacting elements, easily represented by networks.  Communities or clusters of highly connected actors are an essential feature in a multitude of empirical networks,
and identifying these clusters helps answer vital questions in various fields.  Depending on the area of interest,  interacting elements may be metabolites, people, or computers. Their interactions can be represented in chemical reactions, friendships, or some type of communication. For example, a terrorist cell is a cluster in the communication network of terrorists; web pages that provide hyperlinks to each other form a community that may host discussions of a similar topic; a cluster in the network of biochemical reactions might contain metabolites with similar functions and activities. Networks (or graphs) appropriately describe these relationships.  Therefore, the substantive questions in these various disciplines are, in essence, questions regarding the structure of networks.  Given the demonstrated interest in making statistical inference from an observed network, it is essential to evaluate the ability of clustering algorithms to estimate the ``true clusters" in a network model. Understanding when and why a clustering algorithm correctly estimates the ``true communities" provides a rigorous understanding of the behavior of these algorithms and potentially leads to improved algorithms.  



The Stochastic Blockmodel is a model for a random network.  The ``blocks" in the model correspond to the concept of ``true communities" that we want to study.  In the Stochastic Blockmodel, $N$ actors (or nodes) each belong to one of $K$ blocks and the probability of a connection between two nodes depends only on the memberships of the two nodes \citep*{hollandkathryn1983stochastic}.  This paper  adds to the rigorous understanding of the maximum likelihood estimator (MLE) under the Stochastic Blockmodel.  


There has been significant interest in how  various clustering algorithms perform under the Stochastic Blockmodel (for example, \citet*{bickel2009nonparametric, rohe2011spectral, choi2010stochastic, bickel2011method,  zhao2011consistency,celisse2011consistency, channarond2011classification, flynn2012consistent,bickel2012asymptotic, sussman}). In a parallel line of research, several authors have studied clustering algorithms on the Planted Partition Model, a model nearly identical to the Stochastic Blockmodel. For example, \citet*{mcsherry2001spectral} studies a spectral algorithm to recover the planted partition and analyzes the estimation performance of this algorithm. \citet*{chaudhuri2012spectral} improves upon this algorithm by introducing a type of regularization  and proving consistency results under the planted partition model.

In the previous literature, two papers have studied the high dimensional Stochastic Blockmodel, where the number of blocks $K$ grows with the number of nodes $N$ \citep*{rohe2011spectral,choi2010stochastic}.   The impetus for a high dimensional model comes from two empirical observations. First,  \citet*{leskovec2008statistical}   found that in a large corpus of empirical networks, the tightest clusters (as judged by several popular clustering criteria) were no larger than 100 nodes,  even though some of the networks had several million nodes.  This result echoes similar findings in Physical Anthropology.
 \citet*{dunbar1992neocortex}  took various measurements of brain size in 38 different primates and found that the size of the neocortex divided by the size of the rest of the brain had a log-linear relationship with the size of the primate's natural communities.  In humans, the neocortex is roughly four times larger than the rest of the brain.  Extrapolating the log-linear relationship estimated from the 38 other primates, \citet*{dunbar1992neocortex}  suggests that humans do not have the social intellect to maintain stable communities larger than roughly 150 people (colloquially referred to as Dunbar's number).  \citet{leskovec2008statistical}   found a similar result in several other networks that were not composed of humans.   The research of \citet{leskovec2008statistical}   and \citet*{dunbar1992neocortex} suggests that the block sizes in the Stochastic Blockmodel should not grow asymptotically.  Rather, block sizes should remain fixed (or grow very slowly).

In the previous research of \citet*{rohe2011spectral} and \citet*{choi2010stochastic}, the average block size grows at least as fast as $N^{3/4}$ and $N^{1/2}$ respectively.  Even though these asymptotic results allow for $K$ to grow with $N$, $K$ does not grow fast enough.  The average block size quickly surpasses Dunbar's number.   In this paper, we introduce the  highest dimensional asymptotic setting that allows $K = N \log^{-5} N$ and $N/K = \log^{5} N$.  Thus, under this asymptoti setting, the size of the clusters grows much more slowly.  We call it the ``highest" dimensional because, ignoring the $\log$ term, $K$ cannot grow any faster.  If it did, then eventually $K>N$ and there would necessarily be blocks containing zero nodes.  To create a sparse graph, the out-of-block probabilities decay roughly as $\log^\gamma N/N$ in the highest dimensional setting, where $\gamma >0$ is some constant.   To ensure that a block's induced subgraph remains connected, the in-block probabilities are only allowed to decay slowly like $\log^{-1} N$. We show that under this asymptotic setting, a regularized maximum likelihood estimator (RMLE) can estimate the block partition for most nodes.


This paper departs from the previous high dimensional estimators of \citet*{rohe2011spectral} and \citet*{choi2010stochastic} by introducing a restricted parameter space for the Stochastic Blockmodel.  
In several high dimensional settings, regularization restricts the full parameter space providing a path to consistent estimators  \citep*{negahban2010unified}.  If the true parameter setting is  close to the restricted parameter space, then regularization trades a small amount of bias for a potentially large reduction in variance.  
For example, in the high dimensional regression literature, sparse regression techniques such as the LASSO restrict the parameter space to produce sparse regression estimators \citep*{tibshirani1996regression}. Several authors have also suggested parameter space restrictions for high dimensional covariance estimation, e.g. \citet*{fan2008high, friedman2008sparse, ravikumar2011high}.  Parameter space restrictions have also been applied in  Linear Discriminant Analysis 
\citep*{tibshirani2002diagnosis}.  In graph inference, previous authors have explored various ways of incorporating statistical regularization into eigenvector computations \citep*{chaudhuri2012spectral, amini2012pseudo, mahoney2010implementing, perry2011regularized, mahoney2012approximate}.

In this paper, we propose restricting the parameter space for the Stochastic Blockmodel.  These restrictions are supported by empirical observations \citep*{dunbar1992neocortex, leskovec2008statistical}, and they result in a statistically regularized estimator.   We will show that the RMLE is suitable in the highest dimensional asymptotic setting.  This work is distinct from previous approaches to regularization in graph inference because we study a parametric method, the MLE.  

\par

\section{Preliminaries}
\subsection{Highest Dimensional Asymptotic Setting}
In the Stochastic Blockmodel (SBM), each node belongs to one of $K$ blocks.  Each edge corresponds to an independent Bernoulli random variable where the probability of an edge between any two nodes depends only on the two nodes' block memberships \citep*{hollandkathryn1983stochastic}.  The formal definition is as follows.
\begin{definition}
For a node set $\{1,2,...,N\}$, let $P_{ij}$ denote the probability of including an edge linking node $i$ and $j$.   
Let $\tilde{z}$ : $\{1, 2,..., N\} \rightarrow \{1, 2,..., K\}$ partition the $N$ nodes into $K$ blocks. So, $\tilde{z}_i$ equals the block membership for node $i$. $\tilde z$ specifies all true clusters in the model. Let $\boldsymbol{\theta}$ be a $K \times K$ matrix where $\theta_{ab} \in [0, 1]$ for all $a, b$. Then $P_{ij} = \theta_{\tilde{z}_i\tilde{z}_j}$ for any $i,j = 1,2,...,n$.  So under the SBM, the the probability of observing adjacency matrix $A$ is 
\begin{equation*}
  P(A) =\prod_{i<j} {\theta_{ \tilde{z}_i \tilde{z}_j}}^{A_{ij}}{(1-\theta_{ \tilde{z}_i  \tilde{z}_j})}^{(1 - A_{ij})}.
\end{equation*}
The distribution factors over $i<j$ because we only consider undirected graphs without self-loops.
\end{definition}

  The highest dimensional asymptotic setting, defined in Definition \ref{def:hsbm}, restricts the parameters of the SBM in two ways.  First, because empirical evidence suggests that community sizes do not grow with the size of the network,  this setting allows $s$, defined to be the population of the smallest block, to grow very slowly.  
The second restriction ensures that the sampled networks will have sparse edges.  At a high level, there are two types of edges, ``in-block edges" that connect nodes in the same block and ``out-of-block edges" that connect nodes in different blocks.  In order to ensure sparse edges in the high dimensional setting, it is necessary that both the number of out-of-block edges and the number of in-block edges do not grow too fast.   To control the number of out-of-block edges, the off-diagonal elements of $\theta$ must be (roughly) on the order of $1/N$, otherwise the graph will be dense.  The definition allows a set $Q$ to prevent this restriction from becoming too stringent; if $(a,b) \in Q$, then $\theta_{ab}$ is not required to shrink as the network grows, allowing blocks $a$ and $b$ to have a tight connection.  As for the in-block edges, the slowly growing communities prevent these from creating a dense network;  the number of in-block edges connected to each node is bounded by the size of the block population.  As such, the highest dimensional asymptotic setting allows the probability of an in-block connection to remain fixed or decay slowly.  It is necessary to prevent these probabilities from converging to zero too quickly because in such small blocks, it would quickly erase any community structure.

\begin{definition} \label{def:hsbm}
The highest dimensional asymptotic setting is an SBM with the following asymptotic restrictions. 
\begin{enumerate}
\item[(R1)]  For $s$ equal to the population of the smallest block and $x_n = \omega(y_n) \Leftrightarrow  y_n/x_n = o(1)$,  \[s = \omega(\log^\beta N),\ \beta >4.\]
\item[(R2)] Let $(c,d)$ be the interval between $c$ and $d$ and let $Q$ contain a subset of the indices for $\boldsymbol{\theta}$.  For constants $C$ and $f(N) = o(s /\log N)$, 
$$\theta_{ab} = \theta_{ba} \in \left\{ \begin{array}{ll} (\log^{-1} N, 1 - \log^{-1} N)&a=b \\ (1/N^2, Cf(N)/N) & a < b, \{a,b\} \notin Q \\   (\log^{-1} N, 1 - \log^{-1} N)&  a < b, \{a,b\} \in Q. \\ 
\end{array} \right.$$
\end{enumerate}
\end{definition}


Assumption (R1) requires that the population of the smallest block $s = \omega(\log^\beta N), \beta >4$. This includes the scenario where each block size is very small (e.g. $o(\log^5 N)$). In this case, the expected degree for each node is $o(\log^5 N)$.  In the next sections we will introduce the RMLE and then show that it can identify the blocks under the highest dimensional asymptotic setting.

\subsection{Regularized Maximum Likelihood Estimator}

Under the highest dimensional asymptotic setting, 
the number of parameters in $\boldsymbol{\theta}$ is quadratic in $K$ and the sample size available for estimating each parameter in $\boldsymbol{\theta}$ is as small as $s^2$.  For tractable estimation in the ``large $K$ small $s$" setting, we propose an RMLE. 


Recall that  $\tilde{z}$ denotes the true partition.  Let $z$ denote any arbitrary partition.  The log-likelihood for an observed adjacency matrix $A$ under the SBM w.r.t node partition $z$ is
\begin{equation*}
L(A;z, \boldsymbol{\theta}) =\log P(A;z, \boldsymbol{\theta}) = \sum_{i<j}\{A_{ij}\log \theta_{ z_i z_j}+(1 - A_{ij})\log(1-\theta_{ z_i  z_j})\}.
\end{equation*}
For fixed class assignment $z$, let $N_a$ denote the number of nodes assigned to class $a$, and let $n_{ab}$ denote the maximum number of possible edges between class $a$ and $b$; i.e., $n_{ab}$ = $N_aN_b$ if $a  \neq b$ and $n_{aa} = {N_a \choose 2}$.  For an arbitrary partition $z$, the MLE of $\boldsymbol{\theta}$ is 
\[\boldsymbol{\hat{\theta}^{(z)}} = \arg \max_{\boldsymbol{\theta} \in [0,1]^{K \times K}} L(A; z, \boldsymbol{\theta}).\]
This is a symmetric matrix in the parameter space $\Theta = [0,1]^{K \times K}$. 
It is straightforward to show
\begin{equation*}
\hat{\theta}_{ab}^{(z)} = \frac{1}{n_{ab}}\sum_{i<j}A_{ij}1\{z_i = a,z_j = b\}, \quad \forall a,b = 1,2,...,K
\end{equation*}
By substituting $\boldsymbol{\hat{\theta}^{(z)}}$ into $L(A;z, \boldsymbol{\theta})$, we  get the profiled log-likelihood (\citet*{bickel2009nonparametric}).  Define
\begin{equation*}
L(A; z) = L(A; z, \boldsymbol{\hat{\theta}^{(z)}}). 
\end{equation*}
Define $\hat z = \arg \max_z L(A; z)$ as the MLE of $\tilde z$.  To define the RMLE, define the restricted parameter space, $\Theta^R \subset \Theta$, by the following regularization:
\begin{equation*}
\Theta^R = \left\{\boldsymbol{\theta} \in {[0, 1]}^{K \times K}:   \theta_{ab} = c, \ \forall  \ a \neq b \ \mbox{ and for  }  c \in [0,1] \right\}.
\end{equation*}
If $\theta \in \Theta^R$, then all  off-diagonal elements of $\theta$ are equal.  We call the new estimator ``regularized" because, where $\Theta $ has  $K(K+1)/2$ free parameters,  $\Theta^R$ has only $K+1$ free parameters. 

Given class assignment $z$, The RMLE $\boldsymbol{\theta}^{R,(z)}$ is the maximizer of $L(A;z, \boldsymbol{\theta})$ within $\Theta^R$. 
\begin{equation*}
\boldsymbol{\theta}^{R,(z)} = \arg \max_{\boldsymbol{\theta} \in \Theta^R} L(A; z, \boldsymbol{\theta}).
\end{equation*}
The optimization problem within $\Theta^R$ can be treated as an unconstrained optimization problem within ${[0, 1]}^{K+1}$ since we force the off-diagonal elements of $\boldsymbol{\theta}$ to be equal to some number $r$. It has a closed form solution:
\begin{equation*}
\hat{\theta}^{R,(z)}_{ab}  = \left\{ \begin{array}{ll} \hat{\theta}_{aa}^{(z)} = \frac{1}{n_{aa}}\sum_{i<j}A_{ij}1\{z_i = a,z_j = b\}&a=b,\\ \hat{r}^{(z)} = \frac{1}{n_{out}}\sum_{i<j}A_{ij}1\{z_i \neq z_j\}   & a \neq b.  \\ \end{array} \right.
\end{equation*}
Here $n_{out} = \sum_{a<b}n_{ab}$ is the maximum number of possible edges between all different blocks. The Regularized MLE for $\theta_{aa}$ is exactly the same as ordinary MLE, while the Regularized MLE for $\theta_{ab}, a \neq b$ is set to be equal to the total off-diagonal average. 
Finally, by substituting $\hat{\boldsymbol{\theta}}^{R,(z)}$ into $L(A;z, \boldsymbol{\theta})$, define the   regularized profile log-likelihood to be 
\begin{equation*}
L^R(A; z) = L(A; z, \hat{\boldsymbol{\theta}}^{R,(z)}) = \sup_{\boldsymbol{\theta} \in \Theta^R} L(A; z, \boldsymbol{\theta}),
\end{equation*}
and denote the RMLE of the true partition $\tilde{z}$ to be 
\begin{equation} \label{def:rmle}
\hat{z}^{R} = \arg \max_{z} L^R(A; z).
\end{equation}
\par

\section{Performance of the RMLE in the highest dimensional asymptotic setting}

Our main result shows that most nodes are correctly clustered by the RMLE under the highest dimensional asymptotic setting.  This result requires the definition of ``correctly clustered" from  \citet*{choi2010stochastic}. 
\begin{definition} \label{def:Ne}
For any estimated class assignment $z$, define $N_e(z)$ as the number of incorrect class assignments under $z$, counted for every node whose true class under $\tilde{z}$ is not in the majority within its estimated class under $z$.  
\end{definition}
The main result, Theorem \ref{main}, uses the KL divergence  between two Bernoulli distributions.  This is defined as
\[D(p \| q) = p \log \frac{p}{q} + (1-p) \log \frac{1-p}{1-q}.\]
Recall that under the highest dimensional asymptotic setting, $Q$ denotes the off diagonal indices of $\boldsymbol \theta$ that do not asymptotically decay.  Additionally, $n_{ab}$ denotes the total number of possible edges between nodes in block $a$ and nodes in block $b$.   Define $|Q|$ as the number of possible tight edges across different blocks, 
\begin{equation} \label{def:q}
|Q| = \sum_{\{a,b\}\in Q}n_{ab}.
\end{equation}

The following theorem is our main result.  It shows that under the highest dimensional asymptotic setting, the proportion of nodes that the RMLE misclusters converges to zero. 

\begin{theorem} \label{main}
Under the highest dimensional asymptotic setting in Definition \ref{def:hsbm}, $N$ is the total number of nodes, and $s$ is the population of the smallest block.  Assume that the set of friendly block pairs $Q$ (defined in R2 of Definition \ref{def:hsbm}) is small enough that $|Q| = o(Ns)$, where $|Q|$ is defined in Equation \ref{def:q}.  Furthermore, for the matrix of probabilities $\boldsymbol \theta$, assume that for any distinct class pairs $(a, b)$, there exists a class $c$ such that the following condition holds:
\begin{equation}\label{identifyAssumption} 
D\left(\theta_{ac}\| \dfrac{\theta_{ac}+\theta_{bc}}{2}\right) + D\left(\theta_{bc}\| \dfrac{\theta_{ac}+\theta_{bc}}{2}\right) \geq C \dfrac{MK}{N^2}
\end{equation}
Under these assumptions, RMLE $\hat{z}^R$ defined in Equation \ref{def:rmle} satisfies
\begin{equation*}
 \frac{N_e(\hat{z}^R)}{N} = o_p(1),
 \end{equation*}
 where $N_e(z)$ is the number of misclustered nodes defined in Definition \ref{def:Ne}.
\end{theorem}

This theorem requires two main assumptions.  The first main assumption is  $|Q| = o(Ns)$.  Define the number of expected edges $ M = \sum_{i<j} EA_{ij}$.   Under the highest dimensional asymptotic setting, this first assumption implies that $M$ grows slowly, specifically $M =  \omega(N(\log N)^{3+\delta})$, where $\delta >0$.  The second main assumption says that every distinct class pair $(a,b)$ has at least one class $c$ that satisfies Equation \ref{identifyAssumption}.  This assumption relates to the identifiability of $\tilde z$ under the highest dimensional asymptotic setting. For example, if $(a, b) \not \in Q$, then choosing $c = a$ satisfies the assumption in Equation \ref{identifyAssumption}, because $\theta_{aa}$ is large and $\theta_{ba}$ is small.  However, if $(a,b) \in Q$, then there should exist at least one class $c$ to make $\theta_{ac}$,$\theta_{bc}$ identifiable. Otherwise, blocks $a$ and $b$ should be merged into the same block.  Interestingly, this assumption is not strong enough to ensure that $\tilde z$ maximizes $E(L^R(A, \cdot))$, but this is not relevant for our asymptotic results.  If one is concerned about this abnormality, it would be enough to assume in R2 (in the definition of the highest dimensional asymptotic setting) that if $\{a,b\} \in Q$, then  $\theta_{ab} < \Delta$.  This ensures that the probabilities in the set $Q$ are smaller than the in-block probabilities.  Such an assumption does not change the asymptotic result.  

While theorem \ref{main} does not make an explicit assumption about the size of the largest block, Equation \ref{identifyAssumption} makes an implicit assumption because the size of the largest blocks affects the number of edges $M$.  Equation \ref{identifyAssumption} is satisfied when $MK/N^2 \rightarrow 0$ and the set $Q$ does not interfere.  For example, if $|Q| = 0$ and the largest block is $O(N^{1/2-\epsilon})$ for some $\epsilon > 0$, then Equation \ref{identifyAssumption} is satisfied.  

\par

\section{Simulations}
This section compares the RMLE's and the MLE's ability to estimate the block memberships in the Stochastic Blockmodel. In our simulations, the RMLE outperforms the MLE in a wide range of scenarios, particularly when there are several blocks and when the out-of-block probabilities are not too heterogeneous.

\subsection{Implementation}
Computing the exact RMLE and MLE is potentially computationally intractable owing to the combinatorial nature of the parameter space.  In this simulation, we fit the MLE with the pseudo-likelihood algorithm proposed in  \cite*{amini2012pseudo}.  A slight change to the pseudo-likelihood algorithm can fit the RMLE as well; immediately after the pseudo-likelihood algorithm updates $\theta^{(z)}$, we replace the off-diagonal elements with the average of the off-diagonal elements.  

In our simulations, we observed that without any further modification, fitting the RMLE with the adjusted pseudo-likelihood removes blocks from the estimate.  Specifically if data is generated from a model with 30 blocks and the algorithm is given $K = 30$, then the algorithm often returns a partition with only 25 (or so) non-empty sets.  Meanwhile, in the hundreds of simulations that we performed, only two times did the MLE remove a block.  On both of these instances, we re-ran the k-means step in the spectral initialization and gave the MLE this new initialization.  It then returned a partition with $K$ non-empty sets. The RMLE was much more persistent in returning a reduced partition.

Interestingly, even though the resulting RMLE estimate should be be at a disadvantage because it has fewer blocks, it has nearly identical misclustering performance compared to the MLE.  This suggests that the pseudo-likelihood RMLE discards meaningless blocks from the current iterate, a type of dimension reduction.  This provides the opportunity to escape local optima near the initialization.  

To illustrate the benefit of discarding meaningless blocks, the simulations below make a further adjustment to the pseudo-likelihood RMLE.  When the algorithm discards a block, it ``re-seeds" a new block.  This re-seeding is done by the following algorithm that was motivated by follow-up work to the current paper (see \cite{rohe2013blessing}):
\begin{enumerate}
\item Find the block (as defined by the current iteration of the partition) with the smallest empirical in-block probability.
\item For each node in this block, take its neighborhood and remove any nodes that do not connect to any other nodes in the neighborhood.  Call this the transitive neighborhood.
\item Combine both (1) the node with the largest transitive neighborhood and (2) this node's transitive neighborhood into a new block.
\end{enumerate}

Similarly to the suggestion in \cite*{amini2012pseudo}, we initialize the pseudo-likelihood algorithm with spectral clustering using the regularized graph Laplacian \citep*{chaudhuri2012spectral}. Specifically, it runs k-means on the top $K$ eigenvectors of the matrix $D_\tau^{-1/2} A D_\tau^{-1/2}$, where $D_\tau^{-1/2}$ is a diagonal matrix whose $i,i$th element is $1/\sqrt{D_{ii} + \tau}$.  $D_{ii} = \sum_j A_{ij}$ is the degree of node $i$ and tuning parameter $\tau$ is set to be the average degree of all nodes.  

\subsection{Numerical results}

This section contains two sets of simulations.  In the first set of simulations, $K$ is growing while everything else remains fixed.  The second set of simulations investigate the sensitivity of the algorithms to heterogeneous values in the off-diagonal elements of $\theta$.

\begin{figure}[htbp] 
   \centering
   \includegraphics[width=6in]{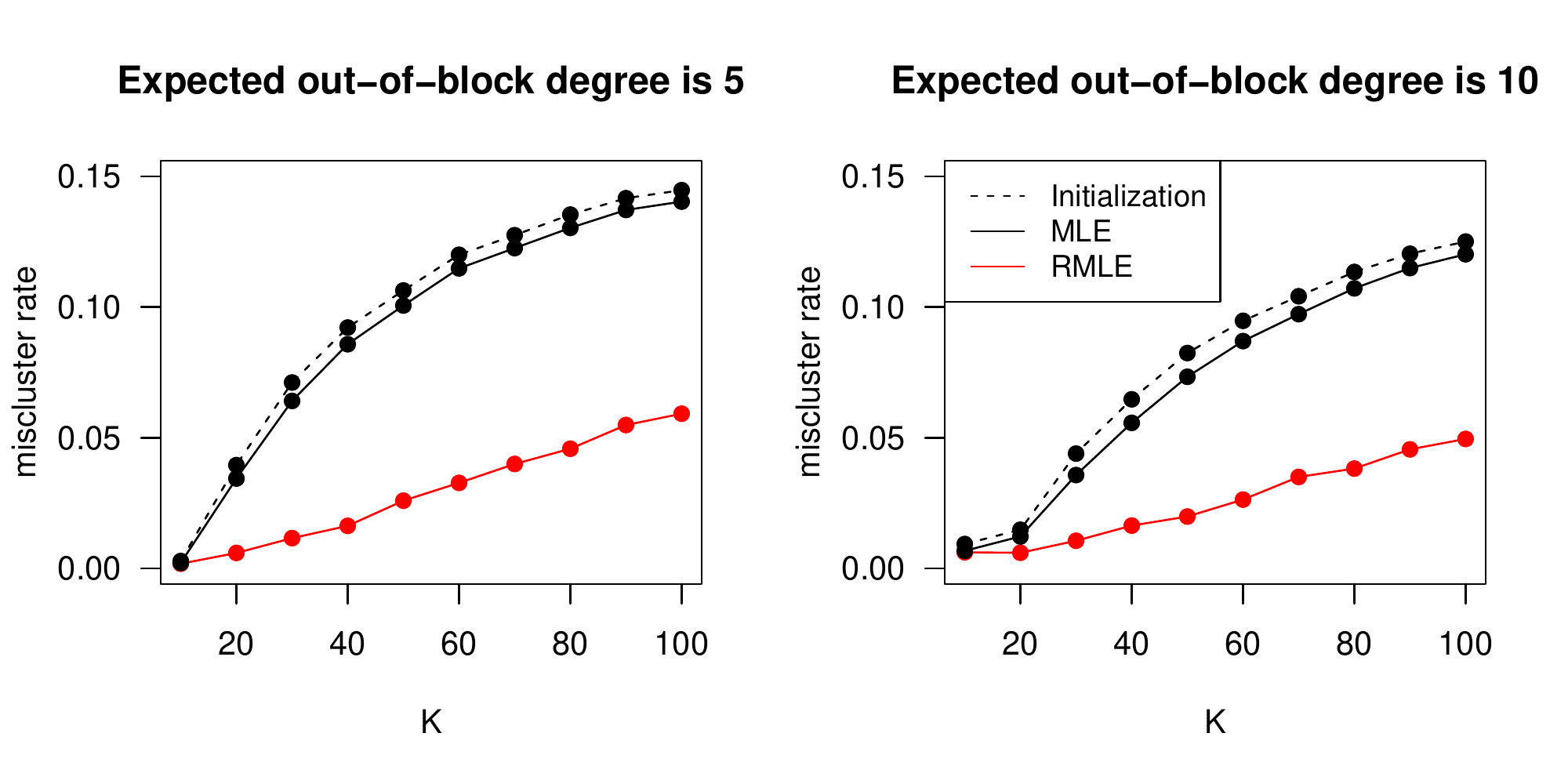}
   \caption{In this simulation, across a wide range of $K$, the RMLE misclusters fewer nodes than the MLE.  In each simulation, every block contains 20 nodes and $K$ grows from 10 to 100 along the horizontal axis.  The vertical axis displays the proportion of nodes misclustered.  Both algorithms are initialized with regularized spectral clustering and the results for this initialization are displayed by the dashed line.  The MLE makes minor improvements to the initialization, while the RMLE makes more significant improvements.  Each point in this figure represents the average of 300 simulations.  All methods were run on the same simulated adjacency matrices.}
   \label{fig:example1}
\end{figure}

The results in Figure \ref{fig:example1} compare the RMLE and MLE under an asymptotic regime that keeps the population of each block fixed at twenty nodes and simply adds blocks.  The horizontal axes corresponds to $K$ growing from ten to one hundred.  In both the left and the right panel, the probability of a connection between two nodes in the same block is $8/20$.  In the left panel, the probability of a connection between two nodes in separate blocks is $5/N$.  In the right panel, it is $10/N$.  Under these two asymptotics, the expected number of ``signal" edges connected to each node is eight, while the expected number of ``noisy" edges is either five or ten.   The vertical axis in both figures is  $N_e(\hat{z})/N$, the proportion of misclustered nodes.

The results in Figure \ref{fig:hetero} examine the sensitivity of the algorithms to deviations from the model in Figure \ref{fig:example1} that makes the off-diagonal elements of $\theta$ equal to one another.    In all simulations, the expected number of ``signal edges" per node is eight, the expected number of ``noisy edges" per node is 5, $s = 20$, and $K = 40$.  On the left side of Figure \ref{fig:hetero}, the off-diagonal elements of $\theta$ come from the Gamma distribution.  In the top left figure, the shape parameter in the Gamma distribution ($\alpha$) varies along the horizontal axis.  While the shape parameter varies, the rate parameter changes to ensure that each node has an expected out-of-block degree equal to five.\footnote{Since $\theta$ is now random, this expectation is taken over both $A$ and $\theta$.}  Under our scaling of the rate parameter, the variance of the Gamma distribution is proportional to $1/\alpha$.  As such, the small values of $\alpha$ make the out-of-block probabilities more heterogeneous, deviating further from the implicit model.  For a point of reference, recall that $\alpha = 1$ gives the exponential distribution.  Our simulations present $\alpha \in (.1, .55)$, more variable than the exponential distribution.  For values of $\alpha$ greater than $.18$, the RMLE outperforms the MLE.   The bottom left plot shows the top left $400\times 400$ submatrix of the adjacency matrix for a simulated example when $\alpha=.18$; the block pattern is clearly recognizable at this level of $\alpha$, suggesting that the RMLE is surprisingly robust to deviations from the implicit model.

The plots on the right side of Figure \ref{fig:hetero} are similar, except the off-diagonal elements of $\theta$ are scaled Bernoulli($p$) random variables.  Note that when $p=1$, this simulation would be identical to a setting in Figure \ref{fig:example1}.  The scaling ensures that the expected out-of-block degree is always five.  Here, the break-even point is around $p = .14$ and the bottom right figure shows the top left $400\times 400$ submatrix of the adjacency matrix for a sample when $p = .14$; the block pattern is clearly recognizable for this level of $p$.  In both of these cases, the RMLE appears robust to deviations from the implied model.  At the same time, for small levels of $p$ and $\alpha$, the MLE misclusters fewer nodes than the RMLE.

\begin{figure}[htbp] 
   \centering
   \includegraphics[width=6in]{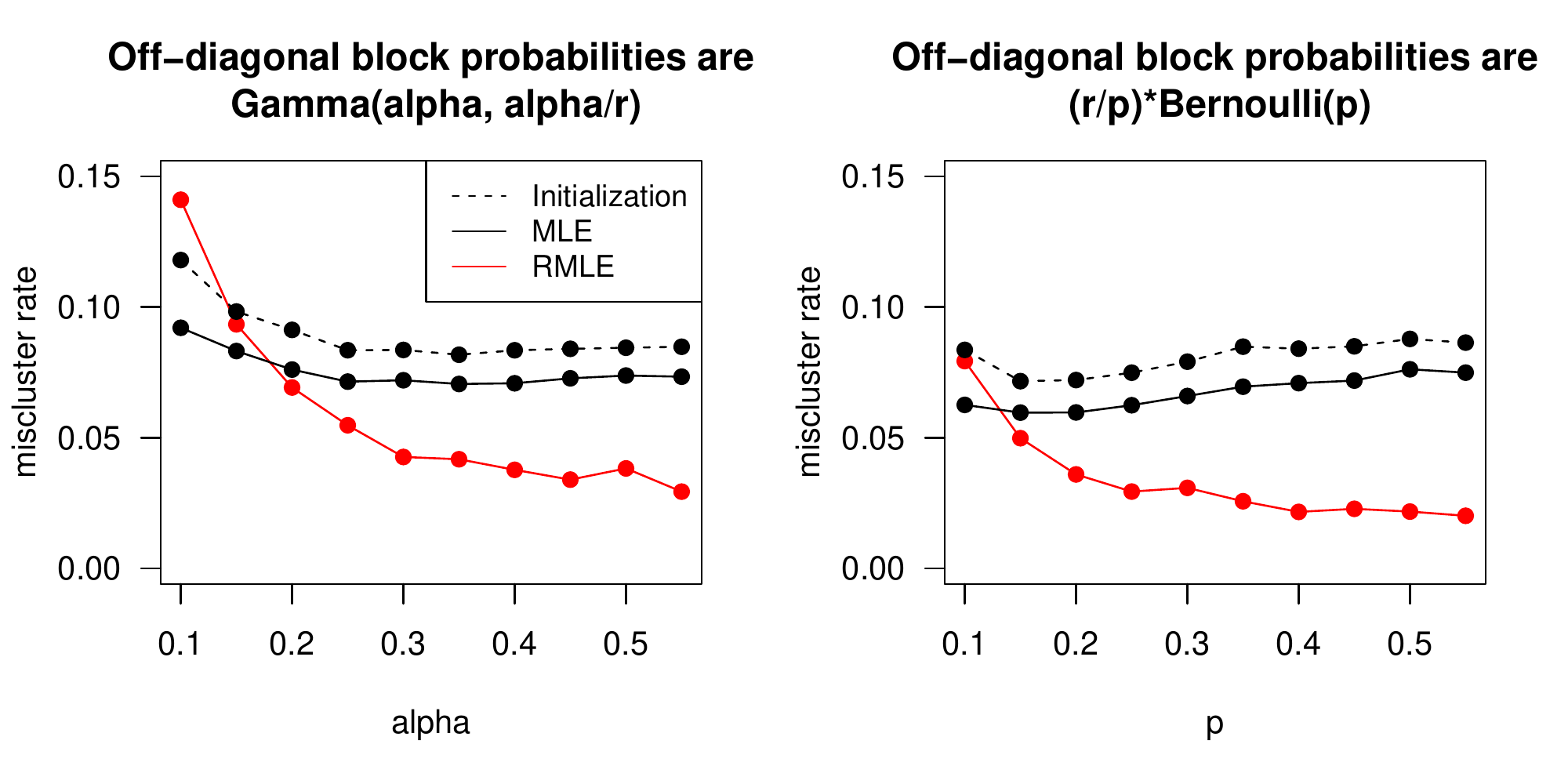}
   \includegraphics[width=6in]{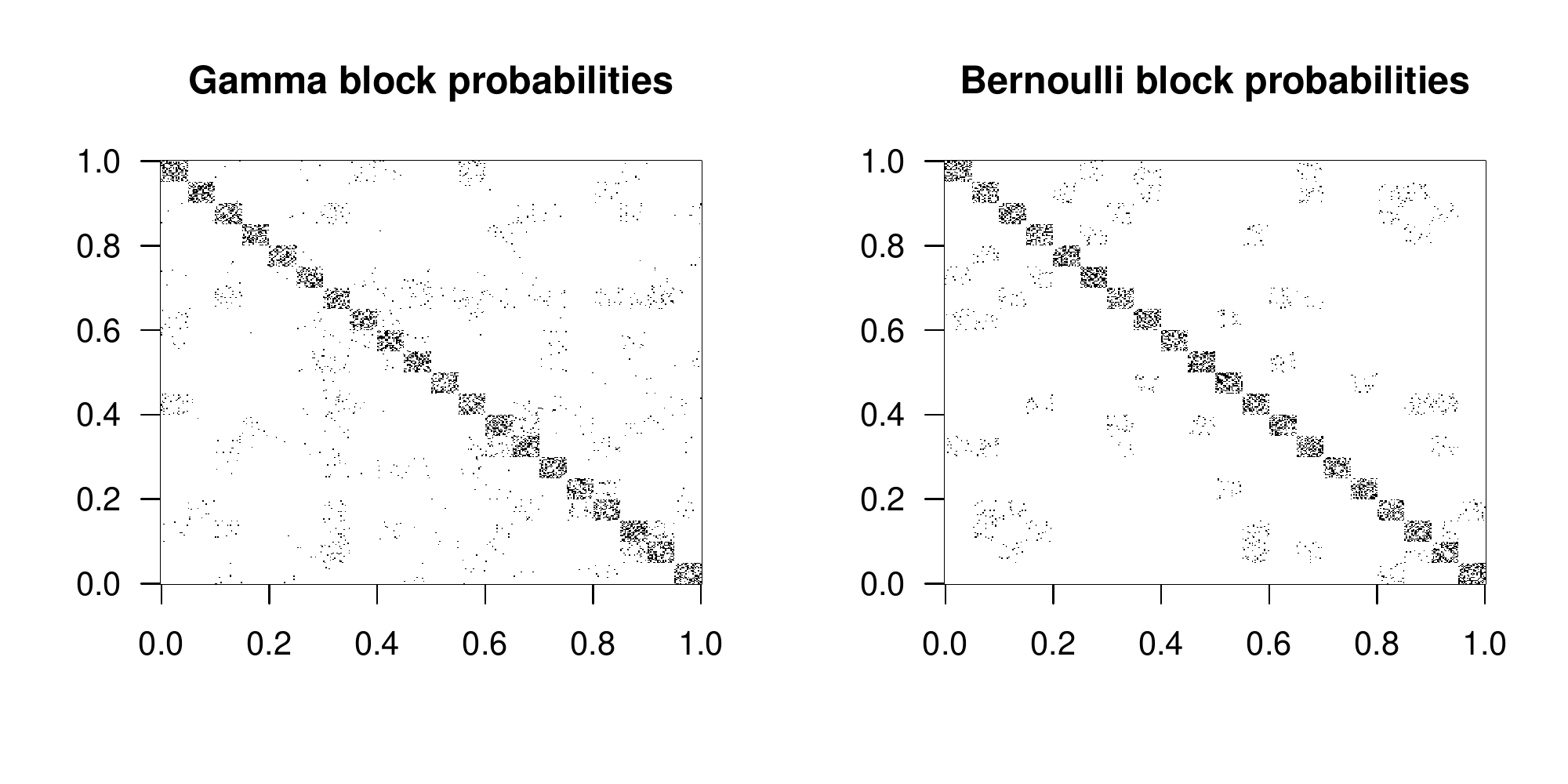}

      \caption{These figures investigate the sensitivity of the algorithms to deviations from the RMLE's ``implied model" that has homogeneous off-diagonal elements in $\theta$.  The top left figure displays results when these elements of $\theta$ come from the Gamma distribution with varying shape parameter. The top right figure displays results when these elements of $\theta$ come from the Bernoulli distribution with varying probability $p$. In both cases, adjustments are made so that each node has five expected out-of-block neighbors.  The bottom plots illustrate the how these heterogenous probabilities manifest in the adjacency matrix;  in both cases, $A$ is sampled with the parameterization that corresponds to the break-even point between the MLE and the RMLE.  Each point represents an average over 200 simulations}
   \label{fig:hetero}
\end{figure}

\section{Discussion}
This paper examines the theoretical properties of the regularized maximum likelihood estimator (RMLE) under the highest dimensional asymptotic setting, showing that under a novel and relevant asymptotic regime, regularization allows for weakly consistent estimation of the block memberships.

%
Under the highest dimensional asymptotic setting, the size of the communities grows at a poly-logarithmic rate, not at a polynomial rate,  aligning with several empirical observations
\citep*{dunbar1992neocortex, leskovec2008statistical}.  There are two natural implications of the block populations growing this slowly.
Under any Stochastic Blockmodel, to ensure the sampled graph has sparse edges, the probability of an out-of-block connection must decay.  In previous ``low-dimensional" analyses, it was also necessary for the probability of an in-block connection to decay.  The first clear implication of small blocks is that the probability of an in-block connection must stay bounded away from zero. Otherwise, a block's induced sub-graph will become disconnected.   The second implication of small block sizes is that the number of off diagonal elements in $\Theta$ grows nearly quadratically with $N$, while the number of in-block parameters (diagonal elements of $\Theta$) grows linearly with $N$.  

The proposed estimator, restricts the parameter space of the SBM in a way that leverages both of these implications. 
%
Since the out-of-block edge probabilities decay to zero, we maximize the likelihood over a parameter space that estimates the probabilities as equal.  Theorem \ref{main} shows that under the highest dimensional asymptotic setting and certain conditions that are similar to identifiability conditions, the RMLE can estimate the correct block for most nodes.  Correspondingly, the simulation section demonstrates the advantages of the RMLE over the MLE.  Overall, this paper represents a first step in applying statistically regularized estimators to high dimensional network analysis in a parametric setting.  Because of the computational issues involved in computing both the MLE and the RMLE, future work will propose a ``local estimator" that (1) incorporates the insights gained from the current analysis and (2) is computationally straight-forward.
\par

\section{Proof of the main result}
The  proof requires some additional definitions.  After giving these definitions, we will outline the proof.

Define the expectation of $\hat{\boldsymbol{\theta}}^{(z)}$ and $\hat{\boldsymbol{\theta}} ^{R,(z)}$ to be $\bar{ \boldsymbol{\theta}}^{(z)}$ and $\bar {\boldsymbol{\theta}}^{R, (z)}$. Define the expectation of $L(A; z, \boldsymbol{\theta})$ to be 
\begin{align*}
\bar{L}_P(z, \boldsymbol{\theta}) =E  [L(A; z, \boldsymbol{\theta})] = \sum_{i<j}\{P_{ij}\log \theta_{z_iz_j}+(1 - P_{ij})\log(1-\theta_{z_iz_j})\}.
\end{align*} 
Let $\bar{L}_P(z)$ to be the maximizer of $\bar{L}_P(z, \boldsymbol{\theta})$ over $\Theta$, and let $\bar{L}^R_P(z)$ to be the maximizer of $\bar{L}_P(z, \boldsymbol{\theta})$ over $\Theta^R$. That is, 
\begin{align}
&\bar{L}_P(z) = \bar{L}_P(z, \bar{\boldsymbol{\theta}}^{(z)}) = \sup_{\boldsymbol{\theta} \in \Theta} \bar{L}_P( z, \boldsymbol{\theta}), \\
&\bar{L}^R_P(z) = \bar{L}_P(z, \bar{ \boldsymbol{\theta}}^{R,(z)}) = \sup_{\boldsymbol{\theta} \in \Theta^R} \bar{L}_P( z, \boldsymbol{\theta}).
\end{align}
The proof of the main theorem is divided into five lemmas. The first step is to bound the difference between $\bar{L}_P(\tilde{z})$ and $\bar{L}_P^R(\hat{z}^R)$ (Lemma~\ref{lem:3}). Lemma~\ref{lem:1} and Lemma~\ref{lem:2} are two building blocks of Lemma~\ref{lem:3} . Lemma~\ref{lem:1} establishes a union bound of  $|L^R(A; z) - \bar L_P^R(z)|$ for any partition $z$. Lemma 2  shows that under the true partition $\tilde{z}$, the expectation of regularized likelihood  is close to the expectation of the ordinary likelihood. Lemma~\ref{lem:3} divides $\bar{L}_P(\tilde{z}) - \bar{L}_P^R(\hat{z}^R)$ into three parts and controls them respectively. We can see this as a bias-variance tradeoff;  we sacrifice some bias $\bar{L}_P(\tilde{z}) - \bar{L}_P^R(\tilde{z})$ to decrease the variance $\max_z |L^R(A; z) -  \bar L_P^R(z)|$.  After Lemma~\ref{lem:3}, it is necessary to develop the concept of regularized refinement, an extension of the refinement idea proposed in \cite{choi2010stochastic}.  Using the concept of regularized refinement, we can bound the error rate $N_e(\hat{z}^R)/N$ with a function of $\bar{L}_P(\tilde{z}) - \bar{L}_P^R(\hat{z}^R)$.  Lemma~\ref{lem:4} and Lemma~\ref{lem:5} use a new concept of regularized  refinement to connect the bounds on the log-likelihood with the error rate $N_e(\hat{z}^R)/N$.   From here on, we write $\hat\theta$ and $\bar\theta$ instead of $\hat\theta^{(z)}$ and $\bar\theta^{(z)}$ when the choice of $z$ is understood.

\begin{lemma} Let $M$ to be the total expected degree of A. That is, $M = \sum_{i<j} EA_{ij}$. \label{lem:1}
\begin{equation}
\max_z |L^R(A; z) - \bar L_P^R(z)| = o_p(M).
\end{equation}
\end{lemma}

This proof follows a similar argument made in \citet*{choi2010stochastic}.
\begin{proof}
Let $H(p) = -p \log p -(1-p)\log(1- p)$, which is the entropy of a Bernoulli random variable with parameter $p$. Define $X = \sum_{i<j} A_{ij} \log\{ \bar{\theta}_{z_iz_j}/(1-\bar{\theta}_{z_iz_j}) \}$. Let $n_{ab}$ denote the maximum number of possible edges between all different blocks.
\begin{align*}
L^R(A; z) - \bar L_P^R(z) &= -\sum_{a = 1}^{K} n_{aa}(H(\hat{\theta}_{aa}) - H(\bar{\theta}_{aa})) - n_{out}(H(\hat{r}) - H(\bar{r}))\\
&= \sum_{a = 1}^{K} n_{aa}D(\hat{\theta}_{aa}\| \bar{\theta}_{aa}) + n_{out}D(\hat{r}\| \bar{r}) + X - E(X).
\end{align*}

For the first part $\sum_{a = 1}^{K} n_{aa}D(\hat{\theta}_{aa}\| \bar{\theta}_{aa})+ n_{out}D(\hat{r}\| \bar{r}) $, by similar argument as in \cite{choi2010stochastic}, we have that for every regularized estimator $\hat{\theta}^R$:

\begin{align*}
pr(\hat{\theta}^R) \le \exp \bigg\{ - \sum_{a = 1}^{K} n_{aa}D(\hat{\theta}_{aa}\| \bar{\theta}_{aa}) - n_{out}D(\hat{r}\| \bar{r}) \bigg\}.
\end{align*}
Let $\hat {\Theta}$ denote the range of $\hat{\theta}^R$ for fixed z. Then the total number of sets of values $\hat{\theta}^R$ can take is $|\hat{\Theta}| = (n_{out}+1)\cdot\Pi_{a=1}^{K} (n_{aa}+1)$. Notice that $\sum_{a=1}(n_{aa}+1) +(n_{out}+1) = \frac{N(N-1)}{2}+K+1$, we have $|\hat{\Theta}| \le (\frac{N(N-1)}{2(K-1)}+1)^{K+1}\le (\frac{N^2}{2K})^{(K+1)}$. 
Then $\forall \epsilon >0$,
\begin{align*}
pr\bigg\{ \sum_{a = 1}^{K} n_{aa}D(\hat{\theta}_{aa}\| \bar{\theta}_{aa})+ n_{out}D(\hat{r}\| \bar{r}) >\epsilon\bigg\}  &\le |\hat{\Theta}|e^{-\epsilon} \le  (\frac{N^2}{2K})^{(K+1)} e^{-\epsilon}\\
&\le \exp \bigg\{2(K+1)\log N - (K+1)\log(2K)-\epsilon\bigg\}.
\end{align*}

 For the second part $X-E(X)$, each $X_{ij} = A_{ij} \log \{ \bar{\theta}_{z_iz_j}/(1-\bar{\theta}_{z_iz_j}) \}$ is bounded in magnitude by $C = 2\log N$. By the following concentration inequality:
 \begin{align*}
 pr\{ |X - E(X)| \ge \epsilon\} \le 2\exp \bigg\{ -\frac{\epsilon^2}{2\sum_{i<j}E(X_{ij}^2)+ (2/3)C\epsilon}\bigg\}.
 \end{align*}
 Here $\sum_{i<j}E(X_{ij}^2) \le 4M \log^2 N$. Finally, by a union bound inequality over all partition $z$, we have:
 \begin{align*}
 pr\{\max_z |L^R(A; z) - \bar L_P^R(z)|\ge 2\epsilon M\} \le& \exp\{ N \log K + 2(K+1)\log N - (K+1)\log(2K) -M\epsilon\} \\
 & + 2\exp \bigg \{ N \log K - \frac{\epsilon^2M}{8\log^2 N+ (4/3)\epsilon \log N} \bigg \}.
 \end{align*}
 Notice that in this asymptotic setting , the total expected degree $ M = \omega(N(\log N)^{3+\delta})$. Then, $\max_z |L^R(A; z) - \bar L_P^R(z)| = o_p(M)$.
 \end{proof}
 
\begin{lemma}\label{lem:2} Under the true partition $\tilde{z}$, $\bar{L}_P(\tilde{z}) - \bar{L}_P^R(\tilde{z}) = o(M).$
\end{lemma}
\begin{proof}
 When $N$ is sufficiently large,
 \begin{align*}
 \bar{L}_P(\tilde{z}) - \bar{L}_P^R(\tilde{z}) &= \sum_{a<b}n_{ab}D(\theta_{ab}\|  \bar{r})= \sum_{a<b, \{a,b\}\in Q} n_{ab}D(\theta_{ab}\|  \bar{r}) + \sum_{a<b,\{a,b\}\notin Q} n_{ab}D(\theta_{ab}\|  \bar{r})\\
 &\le |Q|C_1 + (N(N-1)/2 - \sum_{a=1}^K n_{aa}-|Q|) \frac{Cf(N)}{N} (\log (CNf(N)))\\
 &\le |Q|C_1 + N^2 \frac{Cf(N)}{N} (\log N+\log Cf(N))=o(M).
 \end{align*}
Here $C_1>0$ is some constant. The last equality is due to the fact that $M = \Omega(Ns)$, which is directly implied by Definition 2.2.
 \end{proof}
 
 \begin{lemma} \label{lem:3}
Under the true partition $\tilde z$ and the RMLE $\hat{z}^R$, $\bar{L}_P(\tilde{z}) - \bar{L}_P^R(\hat{z}^R) = o_p(M).$
\end{lemma}
\begin{proof}
First notice that the left hand side is a nonnegative value since $\tilde{z}$ maximizes $\bar{L}_P(\cdot)$ and $\bar{L}_P(\hat{z}^R) \ge \bar{L}_P^R(\hat{z}^R)$.

By adding another positive term, and using Lemma~\ref{lem:1} and Lemma ~\ref{lem:2}:

\begin{align*}
\bar{L}_P(\tilde{z}) - \bar{L}_P^R(\hat{z}^R) & \le \bar{L}_P(\tilde{z}) - \bar{L}_P^R(\hat{z}^R) + L^R(A; \hat{z}^R) - L^R(A,\tilde{z}) \\
&\le |\bar{L}_P(\tilde{z}) - L^R(A,\tilde{z})| +  |\bar{L}_P^R(\hat{z}^R) - L^R(A; \hat{z}^R)| \\
&\le |\bar{L}_P(\tilde{z}) - \bar{L}_P^R(\tilde{z})| + |\bar{L}_P^R(\tilde{z}) - L^R(A,\tilde{z})|+  |\bar{L}_P^R(\hat{z}^R) - L^R(A; \hat{z}^R)| \\
&= o_p(M).
\end{align*}
\end{proof}

To make $N_e(z)$ mathematically tractable, \citet*{choi2010stochastic} introduced the concept of block refinements. The next paragraphs first reintroduce the definition. We then extend this definition to the regularized block refinement. 


\subsection{Partitions and refinements}  
The refinement is the key concept to connect $\bar{L}_P(\tilde{z}) - \bar{L}_P^R(\hat{z}^R)$ with the error rate $N_e(\hat{z}^R)/N$. For this subsection, we first review the concept of partition and refinement.  Then, we give its regularized version. Second, we state the fact that a refinement's log-likelihood is no less than the original partition's log-likelihood. Then, the distance between log-likehood of its (regularized) refinement and log-likelihood of true $\tilde{z}$ can be bounded by the distance between (regularized) log-likelihood of arbitrary $z$ partitions and log-likelihood true $\tilde{z}$.  Finally, the connection between (regularized) refinement log-likelihood and the error rate is established (Lemma~\ref{lem:5}.)

For positive integer $N$, define $[N]$ as the set $\{1, \dots, N\}$.  The partition log-likelihood $\bar L^*_P$ is defined  for any partition $\Pi$ of the indices of a lower triangular matrix, 
\[\Pi: \{(i,j)\}_{i \in [N], j \in [N], i<j} \rightarrow (1, \dots, L). \]  
Define 
\[S_\ell = \left\{(i,j) : \Pi(i,j) = \ell \mbox{ and } i < j\right\} \ \ \mbox{ and } \ \ \bar{\theta_\ell}= |S_\ell|^{-1} \sum_{i<j: \Pi(i,j) = \ell} P_{ij}.\]
The partition log-likelihood is defined as
\[\bar L^*_P(\Pi)= \sum_{i <j} \{ P_{ij} \log\bar\theta_{\Pi(i,j)}+(1-P_{ij}) \log(1- \bar \theta_{\Pi(i,j)})\}.\]

Notice that any class assignment $z$ induces a corresponding partition $\Pi^z$,
\[\Pi^z(i,j) = \ell, \ \mbox{ where } \ell = z_i + (z_j- 1) \cdot K.\]
It is straightforward to show that $ \bar L^*_P(\Pi^z) = \bar L_P(z)$.  

A refinement $\Pi'$ of partition $\Pi$ further divides the partitions in $\Pi$ into subgroups. Formally,
\begin{definition} 
 A \textbf{refinement} $\Pi'$ of partition $\Pi$ satisfies the following condition. 
 \[ \Pi'(i_1, j_1) = \Pi'(i_2, j_2) \ \ \ \Longrightarrow \ \ \ \Pi(i_1, j_1) = \Pi(i_2, j_2), \ \mbox{ for any $i_1 < j_1$ and $i_2 < j_2$.}\]
  \end{definition}
From Lemma A2 in \cite{choi2010stochastic}, 
\begin{equation}\label{refinementineq}
\bar L^*_P(\Pi)\le \bar L^*_P(\Pi')
\end{equation}  This will be essential for for Lemma \ref{lem:5}.
%
%
%
%
%
%

To define $\Pi^*$, a specific refinement of partition $\Pi^z$, we first need to define a set of triples $T$.  The following construction comes directly from \cite{choi2010stochastic}:    
\begin{quote}``For a given membership class under $z$, partition the corresponding set of nodes into subclasses according to the true class assignment $\tilde z$ of each node. Then remove one node from each of the two largest subclasses so obtained, and group them together as a pair; continue this pairing process until no more than one nonempty subclass remains.  Then, terminate.  If pair $(i,j)$ is  chosen from the above procedure, then $z_i=z_j$ and $\tilde{z}_i \neq \tilde{z}_j$."\end{quote}  
Define $C_1$ as the number of $(i,j)$ pairs selected by the above routine.  Notice that at least one of $i$ or $j$ is misclustered.   In fact, $N_e(z)/2 \le C_1 \leq N_e(z)$.  This will be important for Lemma \ref{lem:4} which connects the error rate $N_e(z)/N$ with the refinement. 

Define the set $T$ to contain the triple $(i,j,k)$ if the pair $(i,j)$ was tallied in $C_1$,  and $k \in [N]$ satisfies
\[ D\left(P_{ik}\|  \dfrac{P_{ik}+P_{jk}}{2}\right) + D\left(P_{jk}\|  \dfrac{P_{ik}+P_{jk}}{2}\right) \geq C \dfrac{MK}{N^2}.\]
From assuming Equation \ref{identifyAssumption} , if $(i,j)$ is tallied in $C_1$, then there exists at least one such $k$.  Further, if $z_k = z_\ell$, then $(i,j,\ell)$ is also in $T$.   The set $T$ is essential to defining the refinement partition $\Pi^{*}$ and later the refined regularized partition $\Pi^{*R}$. 

For each $(i,j,k) \in T$, remove $(i,k)$ and $(j,k)$ from their previous subset under $\Pi^z$, and place them into their own, distinct two-element set.  Define the resulting partition as  $\Pi^*$.  Notice that it is a refinement of $\Pi^z$.

\subsection{Regularized partition and regularized refinement}  To extend the analysis to the RMLE, we will define the regularized partition $\Pi^{zR}$ and the associated refinement partition $\Pi^{*R}$.  $\Pi^{zR}$ partitions the nodes into $K+1$ groups;  if $z_i = z_j$, then $\Pi^{zR}(i,j)=z_i$ and if $z_i \ne z_j$, then $\Pi^{zR}(i,j) = K+1$.   It follows from the definition of $\bar L_p^*$ that $\bar L_p^R(z) = \bar L_p^*(\Pi^{zR})$.
 
Construct $\Pi^{*R}$ in the following way: For each $(i,j,k) \in T$, remove $(i,k)$ and $(j,k)$ from their previous subset under $\Pi^{zR}$, and place them into their own, distinct two-element set.  Define the resulting partition as  $\Pi^{*R}$.  Notice that $\Pi^{*R}$ is constructed from $\Pi^{zR}$ in the same way that $\Pi^{*}$ is constructed from $\Pi^{z}$. 
Define $R$ as the set of elements in the off-diagonal block partition that where not removed by the set $T$,
\[R = \big\{ (q , k)  \in [N] \times [N] : z_q \ne z_k , \ (q, x, k) \not \in T, \ (x, q,k) \not \in T, \ \mbox{ for any } x \in [N]\big\}.\]
Notice that $R$ is one group in $\Pi^{*R}$.  Make a refinement $\Pi'$ by subdividing $R$ into ${K \choose 2}$ new groups:
\[\mbox{For } u<v, u \in [K], v \in [K], \mbox{ define } G_{uv} = \left\{ (i,j) \in R : z_i = u, z_j = v \ \mbox{ or } z_i = v, z_j = u\right\}.\]
It follows that $\Pi' = \Pi^*$.  So, $\Pi^*$ is a refinement of $\Pi^{*R}$ and $\Pi^{*R}$ is a refinement for $\Pi^{zR}$.


%
%
%
%
%
%
%
%
%
%
%
%

\begin{lemma}\label{lem:4}(Theorem 3 in \citet*{choi2010stochastic}) For any partition $z$ and $\Pi^*$ being its refinement, if the size of the smallest block $s = \Omega(\frac{MK}{N^2})$, and for any distinct class pairs (a, b), there exists a class c such that Equation \ref{identifyAssumption} holds, then 
\begin{equation}
\bar{L}_P(\tilde{z}) - \bar{L}_P^*(\Pi^*) = \frac{N_e( z)}{N}\Omega(M).
\end{equation}
\end{lemma}
\begin{proof} 
\[\bar L_p(\tilde  z) - \bar L_p^*(\Pi^*)= \sum_{i<j} D(P_{ij}||\bar \theta_{\Pi(i,j)} )=C_1 \Omega\left(s \ \frac{ MK}{N^2}\right) = \frac{N_e(z)}{N} \Omega(M)\]
\end{proof}

\begin{lemma}\label{lem:5} Let $\Pi^{\hat{z}^R}$ be the partition corresponding to $\hat{z}^R$ (the regularized block estimator).  Let ${\Pi'}$ be the refinement of $\Pi^{\hat{z}^R}$, and let $\Pi'^{R}$ be the regularized refinement of $\Pi^{\hat{z}^R}$. 
\begin{equation}
\bar{L}_P(\tilde{z}) - \bar{L}_P^R(\hat{z}^R) \ge  \bar{L}_P(\tilde{z}) - \bar{L}^*_P({\Pi}'^{R}) \ge 
\bar{L}_P(\tilde{z}) - \bar{L}_P^*({\Pi'}).
\end{equation}
\end{lemma}

\begin{proof}
Recall that taking a refinement increases the partition log-likelihood.
The first inequality is due to the fact that ${\Pi}'^R$ is a refinement of the partition $\Pi^{\hat z^R} $.  The second inequality follows from the fact that ${\Pi'}$ is a refinement of ${\Pi}'^R$.  
\end{proof}

\noindent \textbf{Proof of main theorem}: 
The conditions  in Lemma~\ref{lem:4} are satisfied by the highest dimensional asymptotic setting assumption. By Lemma~\ref{lem:3}, \ref{lem:4}, \ref{lem:5}, we have: 
\[o_p(M) = \bar{L}_P(\tilde{z}) - \bar{L}_P^R(\hat{z}^R) \ge \bar{L}_P(\tilde{z}) - \bar{L}_P^*(\Pi') = \frac{N_e(\hat{z}^R)}{N}\Omega(M). \ \mbox{ Hence } \frac{N_e(\hat{z}^R)}{N} = o_p(1).\]

\par


\subsubsection*{Acknowledgments}
Thanks to Sara Fernandes-Taylor for helpful comments.  Research of KR is supported by a grant from the University of Wisconsin.  Research of TQ is supported by NSF Grant DMS-0906818 and NIH Grant EY09946.
\par

\bibliographystyle{plainnat}
\bibliography{cite}

\vskip .65cm
\noindent
Karl Rohe
\vskip 2pt
\noindent
E-mail: karlrohe@stat.wisc.edu
\vskip 2pt
\noindent
Tai Qin
\vskip 2pt
\noindent
E-mail: qin@stat.wisc.edu
\vskip 2pt
\noindent
Haoyang Fan
\vskip 2pt
\noindent
E-mail: haoyang@stat.wisc.edu

\vskip .3cm
\end{document}